\documentclass{amsart}
\usepackage{amsthm}
\usepackage{amssymb, amsmath}%, amsthm}
\usepackage{thmtools}
\usepackage[dvipsnames]{xcolor}
\usepackage{hyperref}
\usepackage{stmaryrd}
\usepackage[all]{xy}
\usepackage{tikz-cd}

\usepackage{leftidx}
\usepackage{colonequals}
\usepackage{cleveref}

\newtheorem{prop}{Proposition}[section]

\newtheorem{lem}[prop]{Lemma}
\newtheorem{cor}[prop]{Corollary}
\newtheorem{thm}[prop]{Theorem}
\theoremstyle{definition}
\newtheorem{defi}[prop]{Definition}
\theoremstyle{remark}

\newtheorem{remar}[prop]{Remark}

\DeclareMathAlphabet{\mathpzc}{OT1}{pzc}{m}{it}

\DeclareMathOperator{\Hom}{Hom}
\DeclareMathOperator{\Mor}{Mor}

\DeclareMathOperator{\Sym}{Sym}

\DeclareMathOperator{\GL}{GL}

\DeclareMathOperator{\Ker}{Ker}

\DeclareMathOperator{\Gal}{Gal}

\DeclareMathOperator{\Spec}{Spec}

\DeclareMathOperator{\Sp}{Sp}
\DeclareMathOperator{\Rep}{Rep}

\DeclareMathOperator{\Def}{Def}
\DeclareMathOperator{\Set}{Set}

\newcommand{\Qp}{\mathbb {Q}_p}

\newcommand{\Aa}{\mathfrak A}

\newcommand{\Fp}{\mathbb F_p}

\newcommand{\mm}{\mathfrak m}

\newcommand{\OO}{\mathcal O}

\newcommand{\nn}{\mathfrak n}

\newcommand{\pp}{\mathfrak p}

\newcommand{\br}[1]{\llbracket #1\rrbracket}

\newcommand{\qq}{\mathfrak{q}}

\newcommand{\alg}{\mathrm{alg}}
\newcommand{\cont}{\mathrm{cont}}

\newcommand{\rhobar}{\overline{\rho}}

\newcommand{\gen}{\mathrm{gen}}

\DeclareMathOperator{\Spf}{Spf}
\newcommand{\rig}{\mathrm{rig}}

\newcommand{\univ}{\mathrm{univ}}

\newcommand{\kappabar}{\overline{\kappa}}

\newcommand{\chara}{\mathrm{char}}

\newcommand{\xbar}{\bar{x}}

\newcommand{\kbar}{\bar{k}}

\newcommand{\PC}{\mathrm{PC}}
\newcommand{\cPC}{\mathrm{cPC}}

\newcommand{\Lbar}{\overline{L}}

\renewcommand{\xbar}{\bar{x}}

\newcommand{\disc}{\operatorname{disc}}

\newcommand{\Thetabar}{\overline{\Theta}}

\newcommand{\colim}{\mathop{\mathstrut\rm colim}\limits}
\newcommand{\eqto}{\xrightarrow{\cong}}

\newcommand{\Cont}{\mathcal C}

\newcommand{\ps}{\mathrm{ps}}

\newcommand{\Rig}{\mathrm{Rig}}
\newcommand{\op}{\mathrm{op}}

\newcommand{\Fhat}{\widehat{F}}
\newcommand{\wTheta}{\widetilde{\Theta}}

\title[Deformations of pseudocharacters]{Deformations of pseudocharacters and Mazur's finiteness condition}
\author{Vytautas Pa\v{s}k\={u}nas and  Julian Quast }
\date{\today.}
\begin{document}

\begin{abstract} We show that  deformation rings $R^{\ps}$ of $G$-pseudocharacters of a profinite group $\Gamma$ are noetherian, when $\Gamma$ satisfies Mazur's finiteness condition. The proof proceeds by reduction to the case when $\Gamma$ is finitely generated, where the result was previously established by the second author.  This enables us to extend our work on
moduli spaces of $R^{\ps}$-condensed representations of a finitely generated profinite group $\Gamma$, to the groups satisfying Mazur's 
finiteness condition. We also show that  the functor from rigid analytic spaces over 
$\Qp$ to sets, which associates to a rigid space $Y$ the set of continuous
$\OO(Y)$-valued $G$-pseudocharacters of $\Gamma$  is representable by a quasi-Stein rigid analytic space, and we study its general properties. 
We expect these results to be useful, when 
studying global Galois representations. 
\end{abstract}
\maketitle

\tableofcontents
\section{Introduction}

If $\Omega$ is any algebraically closed field of characteristic zero then a
semisimple representation $\rho: \Gamma \rightarrow \GL_n(\Omega)$ of a group $\Gamma$ is 
determined by its trace. Anyone who worked out a character table of a finite group
can appreciate that this is a useful result. Conversely, one might wonder when 
a conjugation invariant function from $\Gamma$ to $\Omega$ is a trace of a representation 
of $\Gamma$. Lafforgue's $G$-pseudocharacters, whose definition we recall in \Cref{laf}, provide an answer to these questions 
for any (generalised) reductive group scheme $G$ instead of $\GL_n$ and without restriction on the characteristic of $\Omega$. They were introduced by V.~Lafforgue in
his work \cite{Laf} on the Langlands correspondence in the function field case  in order to associate Galois representations to automorphic forms. Since then the theory 
has been developed further by B\"ockle--Harris--Khare--Thorne \cite{BHKT}, Dat--Helm--Kurinczuk--Moss \cite{DHKM}, Emerson--Morel \cite{emerson2023comparison}, 
Weidner \cite{weidner}, Zhu \cite{zhu_coherent},
 and J.Q. \cite{quast}.

\subsection{Results} Let $L$ be a finite extension of $\Qp$ with a ring of integers $\OO$ and 
residue field $k$. Let $G$ be a (generalised) reductive group scheme over $\OO$, 
let $G^0$ be the neutral component of $G$. Let $\Gamma$ be a profinite group and 
let $\Thetabar$ be a continuous $k$-valued $G$-pseudocharacter. It is shown 
in \cite{quast} that the deformation problem $\Def^{\Gamma}_{\Thetabar}$, parameterising continuous $A$-valued
$G$-pseudocharacters  of $\Gamma$ lifting $\Thetabar$, where $A$ is a local  artinian $\OO$-algebra 
with residue field $k$, is  pro-represented by 
a local profinite $\OO$-algebra $R^{\ps,\Gamma}_{\Thetabar}$ with residue field $k$.

By the reconstruction 
theorem alluded to above there is a continuous representation 
$\rhobar: \Gamma \rightarrow G(\kbar)$, such that the $G$-pseu\-do\-cha\-racter $\Theta_{\rhobar}$ associated 
to $\rhobar$ is equal to $\Thetabar$. Let $K$ be the kernel of $\rhobar$ and 
let $Q$ be a quotient of $\Gamma$ fitting into an exact sequence 
\begin{equation}\label{def_Q_intro}
0 \rightarrow K(p) \rightarrow Q \rightarrow \Gamma/K \rightarrow 0,
\end{equation}
where $K(p)$ is the maximal pro-$p$ quotient of $K$. We show in \Cref{Q_indep} that 
$Q$ does not depend on the choice of $\rhobar$. We may view $\Thetabar$ as a $G$-pseu\-do\-cha\-racter of $Q$. Moreover, any 
continuous pseudocharacter of $Q$ becomes  a continuous pseudocharacter
of $\Gamma$ after composing it with the quotient map. Hence, 
$\Def^Q_{\Thetabar}$ is a closed subfunctor of
$\Def^{\Gamma}_{\Thetabar}$,
which induces a surjection $R^{\ps, \Gamma}_{\Thetabar}\twoheadrightarrow R^{\ps, Q}_{\Thetabar}$. The following theorem
is the technical heart of the paper. 

\begin{thm}[{\Cref{same_ring}}]\label{intro_main} The map
$R^{\ps, \Gamma}_{\Thetabar}\twoheadrightarrow R^{\ps, Q}_{\Thetabar}$
is an isomorphism. In particular, every deformation of $\Thetabar$ factors through $Q$. 
\end{thm} 

A profinite group $\Gamma$ satisfies Mazur's $p$-finiteness condition $\Phi_p$ if for 
every open subgroup $\Gamma'$ of $\Gamma$ the space of continuous 
group homomorphisms $\Hom^{\cont}(\Gamma', \Fp)$ is finite. If $\Gamma$ satisfies $\Phi_p$ then 
it is easy to see that $K(p)$ is a finitely generated pro-$p$ subgroup 
and hence $Q$ is a finitely generated\footnote{Throughout the paper this means that the group in question is finitely generated as a topological group.} profinite group. It was shown 
in \cite{quast} that the ring $R^{\ps, Q}_{\Thetabar}$ is noetherian
and hence we obtain:

\begin{cor}\label{noeth_intro} If $\Gamma$ satisfies Mazur's  condition $\Phi_p$ then $R^{\ps, \Gamma}_{\Thetabar}$ is noetherian. 
\end{cor} 

In \cite{defG} for a finitely generated profinite group $\Gamma$ we have 
defined a functor $X^{\gen, \Gamma}_{\Thetabar}: R^{\ps, \Gamma}_{\Thetabar}\text{-}\alg \rightarrow \Set$,
such that $X^{\gen, \Gamma}_{\Thetabar}(A)$ is the set of 
$R^{\ps}_{\Thetabar}$-condensed 
representations $\rho: \Gamma \rightarrow G(A)$, such that the $G$-pseudocharacter 
$\Theta_{\rho}$ of $\rho$ is equal to the universal deformation $\Theta^u$ specialised along the map $R^{\ps, \Gamma}_{\Thetabar} \rightarrow A$. The condition 
$R^{\ps, \Gamma}_{\Thetabar}$-condensed is a continuity condition on the representation expressed in an algebraic way.  We have shown in \cite{defG} that for finitely generated profinite groups $\Gamma$  the functor $X^{\gen, \Gamma}_{\Thetabar}$ is represented by a finite type $R^{\ps}_{\Thetabar}$-algebra
$A^{\gen, \Gamma}_{\Thetabar}$. \Cref{noeth_intro} allows us to 
define $X^{\gen,\Gamma}_{\Thetabar}$ for profinite groups $\Gamma$ satisfying $\Phi_p$
in the same way as for finitely generated profinite groups. Using 
\Cref{intro_main} we prove: 

\begin{thm}[{\Cref{XgenQ}}] If $\Gamma$ satisfies $\Phi_p$, then $X^{\gen, \Gamma}_{\Thetabar}= X^{\gen, Q}_{\Thetabar}$.
\end{thm} 

Since $Q$ is finitely generated the above theorem allows us to transfer the results concerning $X^{\gen, \Gamma}_{\Thetabar}$ proved for finitely generated profinite groups $\Gamma$ in \cite{defG}, \cite{inf_laf} to the profinite groups satisfying $\Phi_p$. In particular, we obtain

\begin{cor} If $\Gamma$ satisfies Mazur's condition $\Phi_p$, then $X^{\gen, \Gamma}_{\Thetabar}$ is 
represented by a finite type $R^{\ps, \Gamma}_{\Thetabar}$-algebra. 
\end{cor} 

To a representation $\rho: \Gamma \rightarrow G(A)$ one may associate its 
$G$-pseudocharacter $\Theta_{\rho}$. This map induces a morphism 
of schemes $X^{\gen, \Gamma}_{\Thetabar}\rightarrow X^{\ps, \Gamma}_{\Thetabar}$, 
where $X^{\ps, \Gamma}_{\Thetabar}=\Spec R^{\ps, \Gamma}_{\Thetabar}$. The morphism 
is $G^0$-equivariant for the conjugation action on the source and the trivial 
action on the target. Thus it factors through the GIT-quotient 
$X^{\gen, \Gamma}_{\Thetabar}\sslash G^0 \rightarrow X^{\ps}_{\Thetabar}$. 
By appealing to the analogous result proved for finitely generated groups 
$\Gamma$ in \cite{inf_laf} we obtain:

\begin{cor} If $\Gamma$ satisfies $\Phi_p$, then $X^{\gen, \Gamma}_{\Thetabar}\sslash G^0 \rightarrow X^{\ps}_{\Thetabar}$ is a finite adequate homeomorphism in the sense 
of \cite[Definition 3.3.1]{alper}.
\end{cor} 

If $G=\GL_n$ this result was proved by Carl Wang-Erickson \cite{WE_alg} building on the results of Ga\"etan Chenevier on determinant laws \cite{che_durham}. 

Let $\Rig_L$  be the category of rigid analytic spaces over $L$. If $Y\in \Rig_L$ then 
its ring of global sections $\OO(Y)$ carries a natural topology. We let 
$\tilde{X}_G: \Rig_L^{\op} \rightarrow \Set$ be the functor 
which to a rigid analytic space  $Y$ associates the set of 
continuous $G$-pseudocharacters $\Theta$ valued in $\OO(Y)$.

\begin{thm}[{\Cref{main_rigid}, \Cref{points_rigid}}]\label{main_rigid_intro} If $\Gamma$ satisfies Mazur's  condition $\Phi_p$, then 
the functor $\tilde{X}_G$ is representable by a quasi-Stein rigid analytic space
$X_G$. The $\Lbar$-points of $X_G$ are in canonical bijection with $G^0(\Lbar)$-conjugacy classes of continuous $G$-semisimple\footnote{We remind the reader that a representation 
$\rho: \Gamma\rightarrow G(\Omega) $, where $\Omega$ is an algebraically closed field, is $G$-semisimple (or $G$-completely reducible)  provided that whenever
$\rho(\Gamma)$  is contained in 
$P(\Omega)$ for a parabolic subgroup $P$ of $G$, it is contained in $L(\Omega)$ for a Levi subgroup $L$ of $P$. We  refer the reader to \cite[Section 6]{BMR} for more details.}
representations
$\rho: \Gamma \rightarrow G(\Lbar)$. 
\end{thm}

If $\Gamma$ is finitely generated then the result was proved by J.Q. in \cite{quast}.
Our proof proceeds by reduction to the finitely generated case. To carry out the 
reduction in \Cref{pro-p-again} we prove an analog of \Cref{intro_main} for 
deformation rings of continuous $G$-pseudocharacters valued in local fields. 
If $G=\GL_n$ then \Cref{main_rigid_intro} is the main result of \cite{che_durham}.

\subsection{What was known and tried before} Let us  recall the previous related results.
If $G=\GL_n$ then using the work of Emerson--Morel \cite{emerson2023comparison} 
one can work with Chenevier's determinant laws instead of $\GL_n$-pseudocharacters. In this context \Cref{intro_main} is proven by Chenevier in \cite{che_durham}: 
as part of the theory  one gets a Cayley--Hamilton algebra $E$ together with 
a continuous representation $\rho^{\univ}: \Gamma \rightarrow E^{\times}$. One may then show that  $\rho^{\univ}$ factors through $Q$ and deduce \Cref{intro_main} from this, 
see \cite[Lemma 3.8]{che_durham} and \cite[Lemma 5.5]{defG} for details. 

If $G$ is
any generalised reductive group then the Cayley--Hamilton algebra is not available
anymore. However, one could still try to deduce  \Cref{noeth_intro} from the $\GL_n$-case 
by choosing a 
closed immersion $\tau: G\hookrightarrow \GL_n$ and studying the map 
$R^{\ps, \Gamma}_{\tau\circ \Thetabar} \rightarrow R^{\ps, \Gamma}_{\Thetabar}$ 
induced by composition with $\tau$. If the map is finite then one may deduce from 
the $\GL_n$-case 
that $R^{\ps, \Gamma}_{\Thetabar}$ is noetherian. This approach was carried out by 
J.Q. in \cite[Section 5.3]{quast} for classical groups under the assumption $p>2$ in 
the orthogonal cases. However, this strategy runs into difficult invariant theoretic questions in
general. Since $R^{\ps, \Gamma}_{\tau\circ \Thetabar} \rightarrow R^{\ps,\Gamma}_{\Thetabar}$ 
is in general not surjective the authors do not see a way of deducing \Cref{intro_main}
from the $\GL_n$-case.

If $\Gamma$ is a finitely generated profinite group then J.Q. proves \Cref{noeth_intro}
by showing that the tangent space is finite dimensional, which is equivalent to 
showing that $\Def_{\Thetabar}^{\Gamma}(k[\varepsilon])$ is finite. Let us informally explain the underlying difficulty if we drop the assumption that $\Gamma$ is finitely generated. The $G$-pseudocharacter $\Theta$ 
valued in the ring of dual numbers $k[\varepsilon]$ over $k$ consists of $\OO$-algebra homomorphisms 
$\Theta_n: \OO[G^n]^{G^0} \rightarrow \mathcal C(\Gamma^n, k[\varepsilon])$ for 
$n\ge 1$ satisfying certain compatibility relations, where $\mathcal C(\Gamma^n, k[\varepsilon])$ is the set of continuous maps. Proving finiteness assertions such as finite dimensionality of the tangent space amounts to showing that finitely many $\Theta_n$ determine $\Theta$ uniquely. This is feasible when the group is finitely 
generated, but the argument runs into trouble when this assumption is dropped. A similar problem arises in \Cref{main_rigid_intro}. One has to show that for 
an affinoid algebra $A$ a continuous $A$-valued $G$-pseudocharacter  $\Theta$ takes 
values in some formal model of $A$. It is easy to show that each $\Theta_n$ takes 
values in $\mathcal C(\Gamma^n, \mathcal A_n)$ for some formal model $\mathcal A_n$ 
of $A$, but it is not clear that there is a formal model that works for all $\Theta_n$. 
\subsection{What we actually do}
Let us explain the proof of \Cref{intro_main}. We first prove \Cref{intro_main} for finitely generated free profinite groups $\Gamma$. We show that the map between 
deformation rings $R^{\ps}_{\Thetabar} \rightarrow R^{\square}_{\rhobar}$, 
induced by mapping a deformation of $\rhobar$ to its $G$-pseudocharacter, is 
injective. This  key step is carried out in  \Cref{sec_free}. Its proof uses the fact that one may describe 
the moduli space of $G$-pseudocharacters of a discrete finitely generated free group $F$ as a GIT-quotient of a moduli space of representations of $F$ into $G$. We show that if $\Gamma$ is the profinite completion of $F$ then the rings $R^{\ps}_{\Thetabar}$ and $R^{\square}_{\rhobar}$
are naturally isomorphic to completions of the local rings at points $\Theta|_F$ and $\rhobar|_F$ in the 
respective moduli spaces. The assertion about injectivity boils down to commutative algebra, which is carried out in \Cref{sec_GIT}.  We note that if we drop the 
assumption that $\Gamma$ is a  free profinite group then 
it is not known\footnote{Carl Wang-Erickson has pointed out to us that Jinyue Luo 
\cite{jinyue} has constructed an example, where the map fails to be  injective.} whether $R^{\ps}_{\Thetabar}\rightarrow R^{\square}_{\rhobar}$ is injective, even when $G=\GL_n$, see \cite[Remark 1.25]{che_durham}.  

It is easy to see that the universal deformation 
$\rho^{\univ}: \Gamma\rightarrow G(R^{\square}_{\rhobar})$ factors through $Q$ 
and the injectivity of the map $R^{\ps}_{\Thetabar} \rightarrow R^{\square}_{\rhobar}$ implies \Cref{intro_main} for finitely generated  
free profinite groups $\Gamma$. From this one may deduce \Cref{intro_main} 
for all finitely generated profinite groups $\Gamma$. 

If $\Theta$ is a continuous $G$-pseudocharacter of $\Gamma$ valued in a Hausdorff
topological ring $A$ then there is a closed normal subgroup $\Ker(\Theta)$ of $\Gamma$, whose definition we recall in \Cref{sec_ker}, such that $\Theta$ factors through $\Gamma/\Ker(\Theta)$. \Cref{intro_main} holds if and only if $K/\Ker(\Theta^u)$ is a pro-$p$ group, where $\Theta^u$ is the universal deformation of $\Thetabar$. A profinite group is pro-$p$ if and only if 
every closed finitely generated subgroup is pro-$p$. This allows us to prove 
that $K/\Ker(\Theta^u)$ is pro-$p$ for general $\Gamma$ by considering the restrictions of $\Theta^u$ to closed finitely generated 
subgroups of $\Gamma$ and using the fact that we have already proved  \Cref{intro_main}
for such groups. 

\subsection{Motivation}  The Hermite--Minkowski theorem implies that the Galois group $\Gal(F_S/F)$, where
$F$ is a number field and $F_S$ is a maximal extension of $F$ unramified 
outside a finite set of primes $S$ of $F$, satisfies $\Phi_p$.
It is not known if these groups are topologically finitely generated. 
The deformation rings of representations of $\Gal(F_S/F)$ are frequently encountered in arithmetic applications and $\Gamma=\Gal(F_S/F)$ is the main example that we have in mind in this paper.

In \cite{defG} we used the moduli space $X^{\gen, \Gamma}_{\Thetabar}$ to study 
the deformation rings $R^{\square}_{\rho}$ of representations $\rho: \Gamma \rightarrow G(k)$ such that 
$\Theta_{\rho}=\Thetabar$, when $\Gamma$ is absolute Galois group of a 
$p$-adic local field.  We expect the results of this paper to be useful, when
studying analogous questions for $\Gamma=\Gal(F_S/F)$ for a number field $F$. 
\subsection{Acknowledgements} V.P. thanks Ulrich G\"ortz and Marc Levine
for a stimulating discussion regarding \Cref{sec_GIT}. The authors thank Gebhard B\"ockle, Toby Gee and Carl Wang-Erickson for their comments and the referee for their careful reading of the paper. 

The research of J.Q. was funded by the Deutsche Forschungsgemeinschaft (DFG, German Research Foundation) – project number 517234220.

 \section{Lafforgue's pseudocharacters}\label{laf}

In this section we recall the definition of pseudocharacters and their deformation theory.  Until \Cref{sec_pseudo_def}  we let $\OO$ be any commutative ring.

A \emph{generalised reductive group} $G$ over $\OO$ is a 
smooth affine group scheme, that has reductive geometric fibres and such that $G/G^0$ is finite.

\begin{defi}\label{LafPC} Let $\Gamma$ be an abstract group and let $A$ be a commutative $\OO$-algebra. An \emph{$G$-pseudocharacter} $\Theta$ of $\Gamma$ over $A$ is a sequence $(\Theta_n)_{n \geq 1}$ of $\OO$-algebra maps
$$\Theta_n : \OO[G^n]^{G^0} \to \mathrm{Map}(\Gamma^n,A)$$ for $n \geq 1$, satisfying the following conditions\footnote{Here $G$ acts on $G^n$ by $g \cdot (g_1, \dots, g_n) = (gg_1g^{-1}, \dots, gg_ng^{-1})$. This induces a rational action of $G$ on the affine coordinate ring $\OO[G^n]$ of $G^n$. The submodule $\OO[G^n]^{G^0} \subseteq \OO[G^n]$ is defined as the algebraic invariant module of the $G^0$-representation $\OO[G^n]$. It is an $\OO$-subalgebra, since $G$ acts by $\OO$-linear automorphisms.}:
\begin{enumerate}
    \item For each $n,m \geq 1$, each map $\zeta : \{1, \dots, m\} \to \{1, \dots,n\}$, $f \in \OO[G^m]^{G^0}$ and $\gamma_1, \dots, \gamma_n \in \Gamma$, we have
    $$ \Theta_n(f^{\zeta})(\gamma_1, \dots, \gamma_n) = \Theta_m(f)(\gamma_{\zeta(1)}, \dots, \gamma_{\zeta(m)}) $$
    where $f^{\zeta}(g_1, \dots, g_n) = f(g_{\zeta(1)}, \dots, g_{\zeta(m)})$.
    \item For each $n \geq 1$, for each $\gamma_1, \dots, \gamma_{n+1} \in \Gamma$ and each $f \in \OO[G^n]^{G^0}$, we have
    $$ \Theta_{n+1}(\hat f)(\gamma_1, \dots, \gamma_{n+1}) = \Theta_n(f)(\gamma_1, \dots, \gamma_{n-1}, \gamma_n\gamma_{n+1}) $$
    where $\hat f(g_1, \dots, g_{n+1}) = f(g_1, \dots, g_{n-1}, g_ng_{n+1})$.
\end{enumerate}
\end{defi}
We denote the set of $G$-pseudocharacters of $\Gamma$ over $A$ by $\PC_G^{\Gamma}(A)$.
If $f : A \to B$ is a homomorphism of $\OO$-algebras, then there is an induced map $f_* : \mathrm{PC}_{G}^{\Gamma}(A) \to \mathrm{PC}_{G}^{\Gamma}(B)$.
For $\Theta \in \mathrm{PC}_{G}^{\Gamma}(A)$, the image $f_*(\Theta)$ is called the \emph{specialisation} of $\Theta$ along $f$ and is denoted by $\Theta \otimes_A B$.
It is easy to verify that specialisation along $f : A \to B$ commutes with composition with $\varphi$, i.e. $(\varphi \circ \Theta) \otimes_A B = \varphi \circ (\Theta \otimes_A B)$.

The functor $A \mapsto \PC^{\Gamma}_G(A)$ is representable by an $\OO$-algebra \cite[Theorem 3.19]{quast}.
Let us write $\Rep^{\Gamma}_G(A) := \Hom(\Gamma, G(A))$.
The functor $A \mapsto \Rep^{\Gamma}_G(A)$ is also representable by an $\OO$-algebra.
To a representation $\rho : \Gamma \to G(A)$,
we associate a $G$-pseudocharacter $\Theta_{\rho} \in \PC^{\Gamma}_G(A)$ by the formula
\begin{equation}\label{rep_to_pc}
\Theta_{\rho, m}(f)(\gamma_1, \dots, \gamma_m) := f(\rho(\gamma_1), \dots, \rho(\gamma_m))
\end{equation}
for all $m \geq 1$, all $f \in \OO[G^m]^{G^0}$ and all $\gamma_1, \dots, \gamma_m \in \Gamma$.
The map $\rho \mapsto \Theta_{\rho}$ defines a morphism of schemes $\Rep^{\Gamma}_G \to \PC^{\Gamma}_G$,
which is $G^0$-equivariant for the conjugation action on the source and the trivial action on the target and hence induces a morphism from the GIT-quotient
$\Rep^{\Gamma}_G\sslash G^0\rightarrow \PC^{\Gamma}_G$. 

\subsection{Continuous pseudocharacters}
If $\Gamma$ is a topological group, $A$ is a topological ring and, for all $n \geq 1$ the map $\Theta_n$ has image in the set $\Cont(\Gamma^n, A)$ of continuous maps $\Gamma^n \to A$, then we say that $\Theta$ is \emph{continuous}. We denote the set of continuous $G$-pseudocharacters of $\Gamma$ with values in $A$ by $\cPC_G^{\Gamma}(A)$.
\subsection{Kernels}\label{sec_ker}  If $\Theta\in \PC^{\Gamma}_G(A)$ then 
the \textit{kernel} $\Ker(\Theta)$ of $\Theta$ is defined
as the set of all $\delta\in \Gamma$, such that for all $n\ge 1$, all $f\in \OO[G^n]^{G^0}$  and all $\gamma_1, \ldots, \gamma_n\in \Gamma$, we have
\begin{equation}
\Theta_n(f)(\gamma_1, \ldots, \gamma_n \delta)= 
\Theta_n(f)(\gamma_1, \ldots, \gamma_n).
\end{equation}
It is shown in \cite[Lemma 3.11]{quast} that $\Ker(\Theta)$ is a normal subgroup of $\Gamma$. 
\begin{lem}\label{ker_closed}Assume that $\Gamma$ is a topological group, $A$ is a topological ring and $\Theta\in \cPC^{\Gamma}_G(A)$. 
If $A$ is Hausdorff (e.g. profinite) then $\Ker(\Theta)$ is closed in $\Gamma$. 
\end{lem} 
\begin{proof} For  fixed $(\gamma_1, \ldots, \gamma_n)\in \Gamma^n$ and $f\in \OO[G^n]^{G^0}$ the map 
\begin{equation}
\Gamma \rightarrow A, \quad \delta\mapsto \Theta_n(f)(\gamma_1, \ldots, \gamma_n \delta)-
\Theta_n(f)(\gamma_1, \ldots, \gamma_n)
\end{equation}
is continuous. If $A$ is Hausdorff then the 
preimage of $\{0\}$ is a closed subset of $\Gamma$. 
Since $\Ker(\Theta)$ is the intersection of all such subsets 
we obtain the assertion.
\end{proof} 

\subsection{Deformations of $G$-pseudocharacters}\label{sec_pseudo_def} Let $\OO$ be
the ring of integers in a finite extension $L$ of $\Qp$ with residue field $k$.
Let $\Gamma$ be a profinite group and let   $\Thetabar \in \cPC^{\Gamma}_G(k)$.
We recall a deformation problem of $\Thetabar$ introduced in  \cite[Section 5]{quast}.  Let $\Aa_{\OO}$ be the category of artinian local $\OO$-algebras with residue field $k$.
We define the \emph{deformation functor} of $\Thetabar$
$$ \Def_{\Thetabar} : \Aa_{\OO} \to \Set, ~A \mapsto \{\Theta \in \cPC_G^{\Gamma}(A) \mid \Theta \otimes_A k = \Thetabar\} $$
that sends an object $A \in \Aa_{\OO}$ to the set of continuous $G$-pseudocharacters $\Theta$ of $\Gamma$ over $A$ with $\Theta \otimes_A k = \Thetabar$.
It is pro-representable by a  local profinite $\OO$-algebra $R^{\ps}_{\Thetabar}$ with residue field $k$ by
\cite[Theorem 5.4]{quast}.

We will add a superscript $\Gamma$ and write $\Def^{\Gamma}_{\Thetabar}$, $R^{\ps, \Gamma}_{\Thetabar}$ to emphasise that  we are working with pseudocharacters of $\Gamma$, and will drop the superscript if the context is clear. 

Denote by $\Theta^u \in \cPC^{\Gamma}_G(R^{\ps}_{\Thetabar})$ the universal deformation of $\Thetabar$. If 
$A$ is an $R^{\ps}_{\Thetabar}$-algebra we will write $\Theta^u_{|A}$ for 
the specialisation of $\Theta^u$ along $R^{\ps}_{\Thetabar} \rightarrow A$. Since $R^{\ps}_{\Thetabar}$ is profinite and hence Hausdorff, $\Ker(\Theta^u)$ is a closed normal subgroup of $\Gamma$ by \Cref{ker_closed}.

\section{Geometric invariant theory}\label{sec_GIT}

We first recall the set up of \cite{seshadri}. Let $\OO$ be a quasi-excellent ring and let $S=\Spec \OO$. Then $\OO$ is Nagata \cite[\href{https://stacks.math.columbia.edu/tag/07QV}{Tag 07QV}]{stacks-project}
and hence universally Japanese \cite[\href{https://stacks.math.columbia.edu/tag/0334}{Tag 0334}]{stacks-project}. The last assumption is imposed in  \cite{seshadri}. 
Let $G$ be a reductive group scheme over $S$, so that $G$ is an affine group scheme over $S$, $G\rightarrow S$ is smooth and the geometric fibres are connected reductive groups. 
 
Let $V$ be a free $\OO$-module of finite rank $r$ endowed with a $G$-module structure, let $\check{V}=\Hom_{\OO}(V, \OO)$ and let $\Sym(\check{V})$ be the symmetric algebra over $\OO$. The $G$-module structure on $V$ induces an action of $G$ on $\Spec(\Sym(\check{V})) = \mathbb A^r_S.$ Let $X$ be a closed $G$-invariant subscheme of $\Spec (\Sym(\check{V}))$. The $G$-action on $X$ induces an action on $B$, the ring of functions on $X$. The GIT quotient $X\sslash G$ is represented by the ring of invariants $B^G$.
Below we let $R$ be a $B^G$-algebra, which we assume to be a quasi-excellent ring. In the application, both $\OO$ and $R$ will be complete noetherian local rings, and these
are quasi-excellent \cite[\href{https://stacks.math.columbia.edu/tag/07QW}{Tag 07QW}]{stacks-project}.

\begin{lem}\label{finite} If $R$ is a flat $B^G$-algebra or a flat $\OO$-algebra then the ring $(R\otimes_{B^G} B)^G$ is a finite $R$-module. 
\end{lem} 
\begin{proof} If $R$ is a flat $B^G$-algebra then \cite[Lemma 2]{seshadri} implies that 
\begin{equation}
(R\otimes_{B^G} B)^G \cong R \otimes_{B^G} B^G \cong R
\end{equation}
and we are done. 
If $R$ is flat over $\OO$ then \cite[Lemma 2]{seshadri} implies that 
\begin{equation}
(R\otimes_{\OO} \Sym(\check V))^G\cong 
R\otimes_{\OO} \Sym(\check V)^G.
\end{equation}
It follows from 
\cite[Theorem 2 (ii)]{seshadri} that 
$(R\otimes_{B^G} B)^G$ is a finite 
$(R\otimes_{\OO} \Sym(\check V))^G$-module. 
Since the image of $\Sym(\check V)^G$ in $B$ 
is contained in $B^G$, the action of 
$R\otimes_{\OO} \Sym(\check V)^G$ on $(R\otimes_{B^G} B)^G$ factors through the 
map $R\otimes_{\OO} \Sym(\check V)^G\rightarrow R\otimes_{\OO} B^G \rightarrow R$. Hence, 
$(R\otimes_{B^G} B)^G$ is a finite $R$-module.
\end{proof}

\begin{lem}\label{finite2} Let $R$ be a local ring with residue field $\kappa$. Assume that $R$  is either $B^G$-flat or $\OO$-flat. Let $Z$ be a closed non-empty  subscheme of $X\times_{X\sslash G} \Spec R$. If $Z$ is $G$-invariant then every irreducible component 
of $Z$ contains a point $x$ above the closed point of $\Spec R$ such that $\kappa(x)$
is finite over $\kappa$.
\end{lem} 
\begin{proof} \Cref{finite} implies $(R\otimes_{B^G} B)^G$ is a finite $R$-module. 
Since $R$ is a local noetherian ring this implies that every closed subscheme of 
$(X\times_{X\sslash G} \Spec R)\sslash G$ 
contains a point above the closed point of $\Spec R$. 
Since $G$ is connected every irreducible component 
of $Z$ is $G$-invariant by \cite[Lemma 2.1]{BIP_new}. Thus we may assume that $Z$ is irreducible. The image of $Z$ 
in $(X\times_{X\sslash G} \Spec R)\sslash G$
is closed by \cite[Theorem 3 (iii)]{seshadri} and thus $Z$ contains a point
$x$ above the closed point of $\Spec R$. Since $X\rightarrow X\sslash G$ is 
of finite type we may choose $x$ such that $\kappa(x)$ is finite over $\kappa$.
\end{proof} 

\begin{prop}\label{compl_inj} Let $\xbar: \Spec \Omega \rightarrow X$ be a geometric point, such that 
its $G$-orbit in $X\otimes_{\OO}\Omega$ is closed. Let $x$ be its image in $X$ and 
let $y$ be its image in $X\sslash G$. If $X$ is normal then the induced map 
on the completions of local rings $\widehat{\OO}_{X\sslash G, y} \rightarrow 
\widehat{\OO}_{X,x}$ is injective. 
\end{prop}

\begin{proof} Let $\pp$ be the prime ideal of $B^G$ corresponding to $y$ and let $R$ 
be the completion of the localisation $(B^G)_{\pp}$. 
We note that $R$ is  complete local noetherian and hence quasi-excellent and also $B^G$-flat. 
Since $B$ is normal, $B^G$ and $(B^G)_{\pp}$ are normal. Since $\OO$ is quasi-excellent and $B^G$ is of finite type over $\OO$ by \cite[Theorem 2]{seshadri}, $(B^G)_{\pp}$ is quasi-excellent \cite[\href{https://stacks.math.columbia.edu/tag/07QU}{Tag 07QU}]{stacks-project}. Since $(B^G)_{\pp}$ 
is quasi-excellent and normal, its completion $R$ is also normal \cite[\href{https://stacks.math.columbia.edu/tag/0C23}{Tag 0C23}]{stacks-project}, and hence an integral domain.

Let $Q$ be the fraction field of $R$. 
Since $R$ is $B^G$-flat and $Q$ is $R$-flat, $R$ is a subring of $B_R:=R\otimes_{B^G} B$ and $Q$  is a subring of $B_{Q}:=Q\otimes_{B^G} B$. In particular, $X_R:=\Spec B_R$ and $X_Q:=\Spec B_Q$ 
are non-empty. 
Let $Z'$ be an irreducible 
component of $X_Q$
corresponding to a prime ideal $\qq'$ of 
$B_Q$. Then $Z'$ is $G$-invariant by \cite[Lemma 2.1]{BIP_new}. 
Let $Z$ be the closure of $Z'$ in $X_R$. Then $Z$ is a closed  $G$-invariant
subset of $X_R$,  corresponding to the prime ideal $\qq''= \qq'\cap B_R$.

\Cref{finite2} implies that there is $\xbar_1\in Z(\Omega)$ above $y$. Let 
$\xbar_2\in Z(\Omega)$ 
be a point  contained in the closure of the $G$-orbit 
of $\xbar_1$ in $Z\otimes_R \Omega$, such that its $G$-orbit has minimal dimension. Then $\xbar_2\in Z(\Omega)$ lies above $y$ and its $G$-orbit is closed. It follows from \cite[Theorem 3 (ii)]{seshadri} that $\xbar$ and $\xbar_2$ lie in the same $G$-orbit and thus $\xbar\in Z(\Omega)$. 

 Let $\nn$ be a prime ideal of $B_R$ corresponding to the image of $\xbar$ in $X_R$ and  let $\qq$ be the prime ideal of $B$ corresponding to $x$. We have to show that the natural map $R\rightarrow \widehat{B}_{\qq}$ is injective. 
 Since $R$ is the completion of $(B^G)_{\pp}$ the 
completions of $(R\otimes_{B^G} B)_{\nn}$ and 
 $B_{\qq}$ coincide. Thus it is enough to 
show that the map $R\rightarrow ( B_R)_{\nn}$ is injective. 

By construction we have inclusions $R\subset Q\subset B_Q/\qq'\subset \kappa(\qq')$. Thus the composition $R\rightarrow ( B_R)_{\nn}\rightarrow (B_R/\qq'')_{\nn} \subset \kappa(\qq')$ is injective, which implies that $R\rightarrow (B_R)_{\nn}$ is injective. 
\end{proof}

\section{Finitely generated free profinite groups}\label{sec_free}
Let $F$ be a free group on $r$ generators
$\gamma_1, \ldots, \gamma_r$. 
Let $L$ be a finite extension
of $\Qp$ with the ring of integers $\OO$ and residue field $k$. Let $G$ be 
a generalised reductive group scheme over $\OO$.

\begin{prop}\label{finite_inv}The natural map $\Rep^F_G\sslash G^0 \rightarrow \PC^F_G$ is an isomorphism. 
\end{prop}

\begin{proof}
    By \cite[Theorem 3.20]{quast}, $\PC^F_G$ is representable by $\colim\nolimits_{F_m \to F} \OO[G^m]^{G^0}$, where the colimit is indexed by the category of finitely generated free groups $F_m$ together with a homomorphism $F_m \to F$ and homomorphisms of free groups commuting with the homomorphisms to $F$. This category has a terminal object, given by an isomorphism $F_r \eqto F$. It follows, that the colimit is isomorphic to $\OO[G^r]^{G^0}$.
\end{proof}

Let $\Thetabar\in \PC^{F}_{G}(k)$ and
let $\rhobar: F\rightarrow G(\kbar)$ be a $G$-semisimple representation such that $\Theta_{\rhobar}=\Thetabar\otimes_k \kbar$, such 
representation exists and is uniquely determined up to $G^0(\kbar)$-conjugation by \cite[Theorem 3.7]{quast}.
\begin{prop}\label{inj_prop} Let $\mm$ be a maximal ideal of $\OO(\PC^F_G)$
corresponding to $\Thetabar$ and let $\nn$ be a maximal ideal
of $\OO(\Rep^F_G)$ corresponding to $\rhobar$. Then the
morphism in \Cref{finite_inv} induces an injection on completions of local rings
$\OO(\PC^F_G)_{\mm}^{\wedge} \hookrightarrow \OO(\Rep^F_G)^{\wedge}_{\nn}$.
\end{prop} 
\begin{proof} The map 
$\rho\mapsto (\rho(\gamma_1), \ldots, \rho(\gamma_r))$ induces an isomorphism of $\OO$-schemes $\Rep^{F}_G \cong G^r$. Since $G$ is smooth over $\OO$, we deduce that 
$\Rep^F_G$ is smooth over $\OO$. Since $\OO$ is a DVR we deduce that $\Rep^F_G$ is regular and hence normal. Since $\rhobar$ is $G$-semisimple the $G$-orbit of $\xbar\in \Rep^F_G(\kbar)$ corresponding to $\rhobar$ is closed, \cite[Proposition 6.7]{defG}. The assertion follows from \Cref{finite_inv} and 
\Cref{compl_inj} applied with $X=\Rep^{F}_{G}$. 
\end{proof}

There is a finite extension $k'$ of $k$ inside $\kbar$ such that all
$\rhobar(\gamma_i)$ are contained in $G(k')$. Then 
$\rhobar(F)\subset G(k')$ and since $k'$ is a finite field the homomorphism 
$\rhobar: F\rightarrow G(\kbar)$ extends to a continuous homomorphism $\rhobar: \Fhat\rightarrow G(\kbar)$, 
where $\Fhat$ is the profinite completion of $F$.
In particular, $\Thetabar\in \cPC^{\Fhat}_G(k)$.
Let $L'$ be an unramified extension of $L$ with 
residue field $k'$ and let $\OO'$ be the ring of integers in $L$. 
For $A$ in $\Aa_{\OO'}$ we let  
$D^{\square, \widehat{F}}_{\rhobar}(A)$
be the set of continuous representations $\rho:\widehat{F}\rightarrow G(A)$ congruent to $\rhobar$ modulo $\mm_A$. The functor $D^{\square, \widehat{F}}_{\rhobar}$ is pro-represented by a local profinite $\OO$-algebra $R^{\square, \Fhat}_{\rhobar}$ with residue field $k$. We employ the same notational scheme with $F$ instead of $\Fhat$, but drop the continuity condition. 

\begin{cor}\label{inject_free} The restriction maps $\Theta \mapsto \Theta|_F$, $\rho\mapsto \rho|_F$  induce  isomorphisms 
\begin{equation}
\OO(\PC^F_G)_{\mm}^{\wedge}\cong R^{\ps, \Fhat}_{\Thetabar}, \quad \OO(\Rep^F_G)_{\nn}^{\wedge}\cong R^{\square, \Fhat}_{\rhobar}
\end{equation}
and hence the map $\rho\mapsto \Theta_{\rho}$ induces an injection 
$R^{\ps, \Fhat}_{\Thetabar} \hookrightarrow R^{\square, \Fhat}_{\rhobar}$. 
\end{cor}
\begin{proof} Any $A\in \Aa_{\OO'}$ is a finite ring, and hence $G(A)$ is a finite group. 
Thus the restriction map $D^{\square, \Fhat}_{\rhobar}(A)\rightarrow D^{\square, F}_{\rhobar}(A)$, $\rho \mapsto \rho|_{F}$ is bijective, which implies that the induced map $R^{\square, F}_{\rhobar}\rightarrow R^{\square, \Fhat}_{\rhobar}$ is an isomorphism. Since $\OO(\Rep^F_G)$ is noetherian, and $\kappa(\nn)$ is
a finite field, $D^{\square, F}_{\rhobar}$ is pro-represented by $\OO(\Rep^F_G)_{\nn}^{\wedge}$.

Similarly $\Def^F_{\Thetabar}$ is pro-represented 
by $\OO(\PC^F_G)_{\mm}^{\wedge}$. Since $F$ is dense in $\Fhat$ the restriction $\Theta\mapsto \Theta|_F$ induces an injection $\Def^{\Fhat}_{\Thetabar}(A)\hookrightarrow \Def^F_{\Thetabar}(A)$  for all $A\in \Aa_{\OO}$
by \cite[Lemma 3.2]{quast}. This yields a 
surjection 
$\OO(\PC^F_G)_{\mm}^{\wedge}\twoheadrightarrow R^{\ps, \Fhat}_{\Thetabar}$. 

Since the composition 
\begin{equation}
\OO(\PC^F_G)_{\mm}^{\wedge}\twoheadrightarrow R^{\ps, \Fhat}_{\Thetabar}\rightarrow R^{\square, \Fhat}_{\rhobar}\cong \OO(\Rep^F_G)_{\nn}^{\wedge}
\end{equation}
is injective by \Cref{inj_prop} we obtain the assertion.
\end{proof} 

\begin{cor}\label{Q_again}  Let $Q$ be a quotient of $\Fhat$ fitting in the exact sequence
\begin{equation}
0\rightarrow K(p)\rightarrow Q\rightarrow \Fhat/K\rightarrow 0,
\end{equation}
where   $K\subset \Fhat$ is the kernel of $\rhobar$ and $K(p)$ is the maximal pro-$p$ quotient of $K$. Then $R^{\ps, \Fhat}_{\Thetabar}= R^{\ps, Q}_{\Thetabar}$.
\end{cor}
\begin{proof} If  $A\in \Aa_{\OO}$ then 
the kernel of $G(A)\rightarrow G(k)$ is a finite $p$-group. 
Hence, if $\rho_A$ is a deformation of $\rhobar$ to $A$ then 
the homomorphism $\rho_A: \Fhat \rightarrow G(A)$ factors through 
$Q\rightarrow G(A)$. We deduce that the surjection $R^{\square, \Fhat}_{\rhobar}\twoheadrightarrow R^{\square, Q}_{\rhobar}$ is an isomorphism. Since the natural map $R^{\ps, \Fhat}_{\Thetabar}\rightarrow R^{\square, \Fhat}_{\rhobar}$ is injective by \Cref{inject_free}, we deduce that the composition 
\begin{equation}
R^{\ps, \Fhat}_{\Thetabar}\twoheadrightarrow R^{\ps, Q}_{\Thetabar}\rightarrow R^{\square, Q}_{\rhobar}
\end{equation}
is injective. Hence, the surjection $R^{\ps, \Fhat}_{\Thetabar}\twoheadrightarrow R^{\ps, Q}_{\Thetabar}$ is an isomorphism.
\end{proof}

\begin{cor}\label{pro_p_Q} The quotient 
$K/\Ker(\Theta^u)$ is a pro-$p$ group.
\end{cor}
\begin{proof} \Cref{Q_again} 
implies that $\Fhat\twoheadrightarrow \Fhat/\Ker(\Theta^u)$ factors as $Q\twoheadrightarrow \Fhat/\Ker(\Theta^u)$. Thus $K/\Ker(\Theta^u)$ is a quotient 
of $K(p)$ and hence is pro-$p$.
\end{proof} 

\section{The general case}\label{sec_gen}
Let $\Gamma$ be a profinite group and 
let $\Thetabar\in \cPC^{\Gamma}_G(k)$. 
Let $R^{\ps}:= R^{\ps, \Gamma}_{\Thetabar}$ be the universal ring 
representing $\Def_{\Thetabar}$, and 
let $\Theta^u$ be the universal deformation of $\Thetabar$ introduced in \Cref{sec_pseudo_def}. By the reconstruction theorem \cite[Theorem 3.7]{quast} there 
exists a $G$-semisimple representation 
$\rhobar: \Gamma \rightarrow G(\kbar)$,
uniquely determined up to $G^0(\kbar)$-conjugation, such that $\Theta_{\rhobar}=\Thetabar\otimes_k \kbar$. Theorem 3.8 in \cite{quast} implies that $\rhobar$ is continuous. 
We let $K$ be the kernel of $\rhobar$. 
Since the topology on $\kbar$ is discrete, $K$ is an open subgroup of $\Gamma$. Since $\rhobar$ is uniquely determined up to $G^0(\kbar)$-conjugation, $K$ is independent of the choice of $\rhobar$. Moreover, we have inclusions $\Ker(\Theta^u)\subset \Ker(\Thetabar)\subset K$ and $\Ker(\Theta^u)$ is a closed normal subgroup of $\Gamma$.

Let $X\subset \Gamma$ be a set of generators for $\Gamma$ as a topological group. For a subset $Y$ of $X$ we let 
$F_Y$ be the free group generated by $Y$. The homomorphism $F_Y \rightarrow \Gamma$ extends to a continuous homomorphism $\Fhat_Y \rightarrow\Gamma$ from the profinite completion of $F_Y$. 
We let $\rhobar_Y: \Fhat_Y \rightarrow \Gamma\overset{\rhobar}{\longrightarrow}
G(\kbar)$, $K_Y$ the kernel of $\rhobar_Y$, $\Thetabar_Y: = \Thetabar|_{\Fhat_Y}$, $R^{\ps}_Y$ the universal deformation ring representing $\Def_{\Thetabar_Y}$ and $\Theta^u_Y$ the universal deformation of $\Thetabar_Y$.

\begin{lem}\label{factor_thru} Let $Y$ be a subset of $X$. Then the 
map $\Fhat_Y \rightarrow \Gamma/\Ker(\Theta^u)$ factors through 
$\Fhat_Y/\Ker(\Theta^u_Y)\rightarrow \Gamma /\Ker(\Theta^u)$. 
\end{lem}
\begin{proof}The restriction of 
$\Theta^u$ to $\Fhat_Y$ is a deformation of $\Thetabar_Y$ to $R^{\ps}$. Thus 
there is a continuous homomorphism 
$R^{\ps}_Y \rightarrow R^{\ps}$ such 
that 
\begin{equation}
\Theta^u|_{\Fhat_Y} = \Theta^u_Y \otimes_{R^{\ps}_Y} R^{\ps}.
\end{equation}
This implies the assertion.
\end{proof}

\begin{lem}\label{dense_Y} 
Let $N$ be an open normal subgroup of $K$. Then there is a finite subset 
$Y$ of $X$ such that $K_Y\rightarrow K/N$ is surjective.
\end{lem} 
\begin{proof} We are free to replace 
$N$ by an open subgroup, 
which is normal in $K$.
Since $K$ is open in $\Gamma$ 
and hence of finite index, we may 
assume that $N$ is normal in $\Gamma$. Since  $\Gamma/N$ is a finite group there is a finite subset $Y$ of $X$ such that $\Fhat_Y\rightarrow \Gamma/N$ is 
surjective. The preimage of $K/N$ is 
equal to $K_Y$, and hence the assertion follows.
\end{proof}
Let $Q$ be a quotient of $\Gamma$ fitting in the exact sequence
\begin{equation}\label{def_Q}
0\rightarrow K(p)\rightarrow Q\rightarrow \Gamma/K\rightarrow 0,
\end{equation}
where $K(p)$ is the maximal pro-$p$ quotient of $K$.

\begin{lem}\label{Q_indep} Let $\rhobar_1: \Gamma \rightarrow G(\kbar)$ be a continuous representation with $\Theta_{\rhobar_1}=\Theta_{\rhobar}$. 
Let $K_1=\Ker(\rhobar_1)$ and let $Q_1$ be the quotient of $\Gamma$ defined by \eqref{def_Q} with $K_1$ instead of $K$. 
Then $Q=Q_1$.
\end{lem}
\begin{proof} It follows from the reconstruction theorem that $\rhobar_1$ and $\rhobar$ have the same $G$-semi\-simp\-li\-fi\-cation.
Since $\rhobar$ is $G$-semisimple, after replacing $\rhobar$ by a conjugate with an element of $G^0(\kbar)$ we may assume that 
there is a parabolic subgroup $P$ of $G$ with unipotent radical $U$ and Levi $L$ such that 
$\rhobar_1$ takes values in $P(\kbar)$, $\rhobar$ takes values in $L(\kbar)$ and 
$\rhobar(\gamma) = \rhobar_1(\gamma) \pmod{U(\kbar)}$ for all $\gamma\in \Gamma$. Since  
$U(k')$ is a finite $p$-group for all finite extensions $k'$ of $k$, we obtain a surjection 
$\Gamma/K_1\twoheadrightarrow \Gamma/K$ such that the kernel $K/K_1$ is a finite $p$-group. 
Thus $\Gamma/K$ is a quotient of $Q_1$ and $\Gamma/K_1$ is a quotient of $Q$ and the kernels 
in both cases are pro-$p$. Hence, $Q_1$ is a quotient of $Q$ and $Q$ is the quotient of $Q_1$, 
which implies that $Q=Q_1$.
\end{proof} 

\begin{prop}\label{rewrite} Let $N$ be a closed normal subgroup of $\Gamma$ contained in $K$. If for every finite subset $Y\subset X$ 
the map $K_Y \rightarrow \Gamma/N$ factors through the maximal 
pro-$p$ quotient $K_Y(p)$, then $K/N$ is pro-$p$. In particular, 
$\Gamma \twoheadrightarrow \Gamma/N$ factors through 
$Q\twoheadrightarrow \Gamma/N$. 
\end{prop}
\begin{proof} Let $N'$ be an open normal subgroup of $K$ containing $N$. \Cref{dense_Y} implies that there 
is a finite subset $Y$ of $X$ 
such that $K_Y \rightarrow K/N'$ is surjective. Our assumptions imply that this map factors as $K_Y(p)\rightarrow K/N'$. Thus the order of $K/N'$ is
a power of $p$, which implies that 
$K/N$ is pro-$p$. Hence, $K\rightarrow K/N$ factors as $K(p)\rightarrow K/N$, which implies the last assertion.
\end{proof} 

\begin{cor}\label{pro_p} The group $K/\Ker(\Theta^u)$ is pro-$p$.
\end{cor}
\begin{proof} We will deduce the assertion from \Cref{rewrite} applied with $N=\Ker (\Theta^u)$. The map 
$K_Y \rightarrow K/\Ker(\Theta^u)$ 
factors through 
$K_Y/\Ker(\Theta^u_Y)$ by
\Cref{factor_thru}. Since
$K_Y/\Ker(\Theta^u_Y)$ is pro-$p$ by 
\Cref{pro_p_Q} the map $K_Y \rightarrow K/\Ker(\Theta^u)$ factors 
as $K_Y(p)\rightarrow K/\Ker(\Theta^u)$. The assertion follows from \Cref{rewrite}.
\end{proof}

\begin{cor}\label{same_ring} $R^{\ps, \Gamma}_{\Thetabar}= R^{\ps, Q}_{\Thetabar}$.
\end{cor}
\begin{proof}\Cref{pro_p} implies that 
the surjection $\Gamma\twoheadrightarrow \Gamma/\Ker(\Theta^u)$ factors through 
$Q\twoheadrightarrow \Gamma/\Ker(\Theta^u)$, 
which implies the assertion.
\end{proof}

\section{Mazur's finiteness condition}

A profinite group $\Gamma$ satisfies Mazur's  condition $\Phi_p$ if 
$\Hom^{\cont}(\Gamma', \Fp)$ is a finite dimensional $\Fp$-vector 
space 
for every open subgroup 
$\Gamma'$ of $\Gamma$. We keep the notation of the previous section.

\begin{lem}\label{fg} If $\Gamma$ satisfies Mazur's finiteness condition at $p$ then its quotient $Q$ in \Cref{same_ring}
is a finitely generated profinite group.
\end{lem}
\begin{proof} Since
\begin{equation}
\Hom^{\cont}(K(p), \Fp)\cong \Hom^{\cont}(K, \Fp)
\end{equation}
we deduce that $\Hom^{\cont}(K(p), \Fp)$ is a finite dimensional $\Fp$-vector space. This implies that $K(p)$ is finitely generated, \cite[Lemma 2.1]{Gouvea}.
The assertion follows as $\Gamma/K$ is a finite group.
\end{proof}

\Cref{same_ring} and \Cref{fg} imply that 
$R^{\ps, \Gamma}_{\Thetabar}= R^{\ps, Q}_{\Thetabar}$, 
where $Q$ is a finitely generated quotient of $\Gamma$. 
We may use this to transfer the results of 
\cite{quast}, \cite{defG} and \cite{inf_laf} 
proved for finitely generated profinite groups to the groups
satisfying Mazur's finiteness condition at $p$.

\begin{cor}\label{Phi_p_noeth} If $\Gamma$ satisfies Mazur's  condition 
$\Phi_p$ then $R^{\ps, \Gamma}_{\Thetabar}$ is noetherian.
\end{cor}
\begin{proof} The assertion follows from the corresponding 
result for finitely generated profinite groups \cite[Theorem 5.7]{quast}.
\end{proof}

\subsection{The moduli space  of condensed representations}
Let $R$ be a profinite noetherian $\OO$-algebra. Then any finitely generated 
$R$-module $M$ carries a unique Hausdorff topology, which makes it into a topological 
$R$-module. Let $A$ be an $R$-algebra. A representation $\rho: \Gamma\rightarrow G(A)$
is called \textit{$R$-condensed} if there exists a closed immersion of $\OO$-schemes
$\tau: G\hookrightarrow \mathbb A^n$ and a finitely generated $R$-submodule $M$ of 
$\mathbb{A}^n(A)=A^n$ such that $\tau(\rho(\Gamma))\subset M$ and 
the map $\tau\circ \rho: \Gamma \rightarrow M$ is continuous for the canonical topology on $M$ as a finitely generated $R$-module. We check in 
\cite[Lemma 4.2]{defG} that if $\rho$ is $R$-condensed with respect to 
one closed immersion $\tau$ then the same holds for all closed immersions.
Moreover, in \cite[Section 4.2]{defG} we show that $\rho$ is $R$-condensed 
if and only if there exists a homomorphism of condensed groups $\underline{\Gamma} \to G(A_{\disc} \otimes_{R_{\disc}} \underline{R})$ which recovers $\rho$ after evaluation at a point.
Here $\underline{-}$ denotes the functor from topological spaces to condensed sets and $(-)_{\disc}$ takes a set to a discrete condensed set. This interpretation in terms 
of condensed mathematics of Clausen and Scholze \cite{ScholzeCond} motivates the terminology. Informally,
$R$-condensed should be thought of as a continuity condition on $\rho$, even if 
there is no topology on the $R$-algebra $A$. 

If $\Gamma$ satisfies Mazur's condition $\Phi_p$, then $R^{\ps}_{\Thetabar}$ is 
noetherian by \Cref{Phi_p_noeth}.
Thus we may apply the discussion above with $R=R^{\ps}_{\Thetabar}$ and for an $R^{\ps}_{\Thetabar}$-algebra $A$ we let $X^{\gen}_{\Thetabar}(A)$ be the set of $R^{\ps}_{\Thetabar}$-condensed representations $\rho: \Gamma \rightarrow G(A)$ such that $\Theta_{\rho}=\Theta^u\otimes_{R^{\ps}_{\Thetabar}} A$. 

Let $Q$ be the quotient of $\Gamma$ defined in \eqref{def_Q}. We will add 
a superscript to indicate  which group we are working with. Since $R^{\ps, 
\Gamma}_{\Thetabar}= R^{\ps, Q}_{\Thetabar}$ by \Cref{same_ring}, we have 
an inclusion $X^{\gen, Q}_{\Thetabar}(A)\subset X^{\gen, \Gamma}_{\Thetabar}(A)$.

\begin{prop}\label{XgenQ} If $\Gamma$ satisfies Mazur's condition $\Phi_p$ then 
$X^{\gen, Q}_{\Thetabar}(A)=X^{\gen, \Gamma}_{\Thetabar}(A)$.
\end{prop}
\begin{proof} We choose a closed immersion of $\OO$-group schemes
$\tau: G \hookrightarrow \GL_d$. Composition with $\tau$ induces 
injections $X^{\gen, \Gamma}_{\Thetabar}(A) \hookrightarrow X^{\gen, \Gamma}_{\tau\circ \Thetabar}(A)$ and $X^{\gen, Q}_{\Thetabar}(A) \hookrightarrow X^{\gen, Q}_{\tau\circ \Thetabar}(A)$ by \cite[Proposition 3.8]{defG}. Thus we may assume that $G=\GL_d$. 
In this case the assertion follows from \cite[Lemma 5.5]{defG}.
\end{proof} 

The proposition  allows us to transfer the results proved for finitely generated 
profinite groups to profinite groups satisfying Mazur's condition $\Phi_p$. 
In the corollaries below we cite the reference for the statement, when $\Gamma$ 
is finitely generated. 

\begin{cor} If $\Gamma$ satisfies Mazur's condition $\Phi_p$ then the 
functor $X^{\gen}_{\Thetabar}: R^{\ps}_{\Thetabar} \rightarrow \Set$ 
is represented by a finite type $R^{\ps}_{\Thetabar}$-algebra $A^{\gen}_{\Thetabar}$. 
\end{cor} 
\begin{proof} \cite[Proposition 8.3]{defG}.
\end{proof}

Let $X^{\ps}_{\Thetabar}:=\Spec R^{\ps}_{\Thetabar}$. The conjugation action of $G^0(A)$ on $X^{\gen}_{\Thetabar}(A)$ induces
an action of $G^0$ on $X^{\gen}_{\Thetabar}$.
Mapping $\rho\in X^{\gen}_{\Thetabar}(A)$ to its 
$G$-pseudocharacter induces a $G^0$-equivariant
morphism $X^{\gen}_{\Thetabar} \rightarrow X^{\ps}_{\Thetabar}$ for the trivial action on the 
target. Hence the morphism factors through the 
GIT-quotient $X^{\gen}_{\Thetabar}\sslash G^0 \rightarrow X^{\ps}_{\Thetabar}$.

\begin{cor} If $\Gamma$ satisfies Mazur's condition $\Phi_p$ then the 
natural map 
$$X^{\gen}_{\Thetabar}\sslash G^0 \rightarrow X^{\ps}_{\Thetabar}$$
is a finite adequate homeomorphism in the sense of \cite[Definition 3.3.1]{alper}.
\end{cor} 
\begin{proof} \cite[Theorem 4.12]{inf_laf}.
\end{proof}

\section{Completions at maximal ideals and deformation problems}

Let $\kappa$ be a local field that is an $\OO$-algebra.  Let $\OO_{\kappa}$ be its ring of integers and let $k'$ be the residue field. 
We always equip $\kappa$ with the topology induced by the valuation. Let $\OO'=\OO\otimes_{W(k)} W(k')$, where $W(k)$ is the Witt ring of $k$.

In this section we will study deformation theory of $\Theta\in \cPC^{\Gamma}_G(\kappa)$, where  $\Gamma$ is an arbitrary profinite group and $G$ is a
generalised reductive group scheme over $\OO$.  If $\chara(\kappa)=0$ then we let $\Lambda=\kappa$ equipped with its natural topology and we let $\Lambda_0=\OO_{\kappa}$. 
If $\chara(\kappa)=p$ then\footnote{The study of deformation problem in this setting is motivated by the work of B\"ockle--Juschka \cite[Section 4.7]{BJ_new}.} we let $\Lambda= \OO\otimes_{W(k)} C$, where $C$ 
is a Cohen ring of $\kappa$, which means that $C$ is a complete DVR with 
uniformiser $p$ and residue field $\kappa$. If we choose an isomorphism 
$k'\br{t}\cong \OO_{\kappa}$ of local $k'$-algebras then $C$ can 
be chosen to be the $p$-adic completion of $W(k')\br{t}[1/t]$. There is a natural topology on $\Lambda$, such that the quotient topology on $\kappa$ induced via  $\Lambda \twoheadrightarrow \kappa$ coincides with the valuation 
topology on $\kappa$. The topology on $\Lambda$ is Hausdorff, see \cite[Section 3.5]{BIP_new} for more details. We let $\Lambda_0=\OO'\br{t}$. In both cases we let $\Aa_{\Lambda}$ be the 
category of local artinian $\Lambda$-algebras with residue field $\kappa$. 
Any $A\in \Aa_{\Lambda}$ is a finitely generated $\Lambda$-module, and 
hence carries a unique Hausdorff topology, which makes  $A$ into a topological $\Lambda$-module. 

For $A\in \Aa_{\Lambda}$ we let $\Def_{\Theta}(A)$ be the set of 
$\wTheta\in \cPC^{\Gamma}_G(A)$ such that $\wTheta\otimes_{A} \kappa = \Theta$. 
It follows\footnote{If $\chara(\kappa)=p$ then $\Lambda$ is taken to be equal to $\kappa$
in \cite{quast}. However,  the argument used in the  proof of \cite[Theorem 5.4]{quast} goes through in our more general setting.} from \cite[Theorem 5.4]{quast} that the functor $\Def_{\Theta}: \Aa_{\Lambda}\rightarrow \Set$ is pro-represented by a Hausdorff topological local $\Lambda$-algebra $R^{\ps}_{\Theta}$, which is a projective limit of rings in $\Aa_{\Lambda}$. We let $\wTheta^u\in \cPC^{\Gamma}_G(R^{\ps}_{\Theta})$ be the the universal 
deformation of $\Theta$.

\begin{lem}\label{values} If $\Theta\in \cPC^{\Gamma}_G(\kappa)$ then $\Theta$ takes values in $\OO_{\kappa}$.
\end{lem}
\begin{proof} By the reconstruction theorem \cite[Theorem 3.7]{quast} there 
exists a continuous representation $\rho: \Gamma \rightarrow G(\kappabar)$, such 
that $\Theta_{\rho}= \Thetabar\otimes_{\kappa} \kappabar$. If $\kappa'$ is a
finite extension of $\kappa$ contained in $\kappabar$ then 
$G(\kappa')$ is a closed subgroup of $G(\kappabar)$ and we let 
$\Gamma_{\kappa'}$ be its preimage in $\Gamma$. Then $\Gamma_{\kappa'}$
is a closed and hence profinite subgroup of $\Gamma$. It follows from 
\cite[Lemma 4.3]{cotner} that there exists a finite extension 
$\kappa''$ of $\kappa'$ contained in $\kappabar$ and $g\in G^0(\kappa'')$ 
such that $g \rho(\Gamma_{\kappa'}) g^{-1} \subset G(\OO_{\kappa''})$. 
This implies that $\Theta|_{\Gamma_{\kappa'}}$ takes values in $\OO_{\kappa''}$. 
Since $\Gamma= \bigcup_{\kappa'} \Gamma_{\kappa'}$ we deduce that 
$\Theta$ takes values in $\OO_{\kappabar} \cap \kappa = \OO_{\kappa}$.
\end{proof}

Let $\Thetabar\in \cPC^{\Gamma}_G(k')$ be $\Theta\otimes_{\OO_{\kappa}} k'$.
Let $R^{\ps}_{\Thetabar}$ be the deformation ring representing $\Def_{\Thetabar}:\Aa_{\OO'} \rightarrow \Set$ and let $\Theta^u\in \cPC^{\Gamma}_G(R^{\ps}_{\Thetabar})$ be the universal deformation of $\Thetabar$. Since $\Theta$ is a deformation 
of $\Thetabar$ to $\OO_{\kappa}$ we obtain a homomorphism of local 
$\OO'$-algebras $\varphi: R^{\ps}_{\Thetabar} \rightarrow \OO_{\kappa}$.
The natural map $\Lambda\otimes_{\OO'} R^{\ps}_{\Thetabar}\rightarrow \kappa$
of $\Lambda$-algebras is surjective and we denote its kernel by $\qq$.

\begin{prop}\label{fg_kappa} If $\Gamma$ is finitely generated then there is a natural
isomorphism of local $\Lambda$-algebras between $R^{\ps}_{\Theta}$ and 
the $\qq$-adic completion of $\Lambda\otimes_{\OO'} R^{\ps}_{\Thetabar}$.
In particular, $R^{\ps}_{\Theta}$ is noetherian and $\wTheta^u= \Theta^u\otimes_{R^{\ps}_{\Thetabar}} R^{\ps}_{\Theta}$.
\end{prop}
\begin{proof} Let $B$ be the $\qq$-adic completion of $\Lambda\otimes_{\OO'} R^{\ps}_{\Thetabar}$. Then $B$ is a complete local noetherian $\Lambda$-algebra
with residue field $\kappa$ by \cite[Lemma 3.3.5]{BJ_new}, \cite[Lemma 3.36]{BIP_new}.
Thus $B/\mm_B^n \in \Aa_{\Lambda}$ for all $n\ge 1$ and $\wTheta_B:= \Theta^u\otimes_{R^{\ps}_{\Thetabar}} B$ is a deformation of $\Theta$ to $B$. 

The ring $B$ pro-represents a subfunctor $\Def'_{\Theta}$ of $\Def_{\Theta}$, such 
that if $A\in \Aa_{\Lambda}$, then $\wTheta\in \Def_{\Theta}(A)$ lies in $\Def'_{\Theta}(A)$ if and only if 
$\wTheta$ takes values in a $\Lambda_0$-subalgebra $A_0$ of $A$, which is finite 
as a $\Lambda_0$-module. 
Indeed, in this case $\wTheta$ is a deformation of $\Thetabar$ to $A_0$, and 
since $A_0$ is a limit of rings in $\Aa_{\OO'}$ we obtain a morphism of local 
$\OO'$-algebras $R^{\ps}_{\Thetabar}\rightarrow A_0$, which induces 
a morphism $B\rightarrow A$, such that $\wTheta= \wTheta_B\otimes_B A$. Moreover, 
$\wTheta_B \otimes_B B/\mm_B^n$ takes values in 
the image of $\Lambda_0\otimes_{\OO'} R^{\ps}_{\Thetabar}$ in 
$B/\mm_B^n$, which is a finite $\Lambda_0$-module. 

Since $\Gamma$ is finitely generated as a topological group, we may choose a dense 
discrete finitely generated subgroup $\Sigma$ of $\Gamma$. Since $\Sigma$ is dense 
the restriction map induces an injection $\Def_{\Theta}(A) \subset \PC^{\Sigma}_G(A)$.
Let $A_{\infty}$ be the $\Lambda_0$-subalgebra of $A$ equal to the preimage of 
$\OO_{\kappa}$ under $A\twoheadrightarrow \kappa$. Since $\mm_A$ nilpotent 
and a finite $\Lambda$-module, $A_{\infty}$ is a union of $\Lambda_0$-subalgebras 
$A_n$ of $A$, $n\ge 1$, such that each $A_n$ is a finite $\Lambda_0$-module and 
the image of $A_n \rightarrow \kappa$ is equal to $\OO_{\kappa}$. 
\Cref{values} implies that the restriction to $\Sigma$ identifies 
$\Def_{\Theta}(A)$ with a subset of $\PC^{\Sigma}_G(A_{\infty})$. 
Since $\Sigma$ is finitely generated, $\PC^{\Sigma}_G$ is represented 
by a finite type  $\OO$-algebra by \cite[Proposition 3.21]{quast}. Hence, $\PC^{\Sigma}_G(A_{\infty})$ 
is the union of $\PC^{\Sigma}_G(A_n)$ for $n\ge 1$. This implies that 
$\Def'_{\Theta}(A)=\Def_{\Theta}(A)$, which implies the assertion.
\end{proof}

As in \Cref{sec_gen}, let $\rhobar: \Gamma \rightarrow G(\kbar)$ be a 
continuous, $G$-semisimple representation with $\Theta_{\rhobar}=\Thetabar$.
Let $K$ be the kernel of $\rhobar$ and let $Q$ be the quotient of $\Gamma$ 
defined in \eqref{def_Q}.

\begin{cor}\label{pro-p-again} If $\Gamma$ is finitely generated then the group $K/\Ker(\wTheta^u)$ is 
pro-$p$ and $R^{\ps, \Gamma}_{\Theta}=R^{\ps, Q}_{\Theta}$.
\end{cor}
\begin{proof} It follows from \Cref{fg_kappa} that $\Ker(\Theta^u)\subset \Ker(\wTheta^u)$. Since $K/\Ker(\Theta^u)$ is pro-$p$ by \Cref{pro_p}, 
$K/\Ker(\wTheta^u)$ is also pro-$p$. Thus $\Gamma/\Ker(\wTheta^u)$ 
is a quotient of $Q$, which implies the equality of deformation rings.
\end{proof}

Let $X$ be a set of generators of $\Gamma$. For every subset $Y$ of $X$ the notation 
$K_Y$, $\Thetabar_Y$, $\widehat{F}_Y$, $F_Y$ has the same meaning as in the previous 
section. We let $\Theta_Y$ be the restriction of $\Theta$ to $\widehat{F}_Y$, 
$R^{\ps}_{\Theta_Y}$ the universal deformation ring of $\Theta_Y$ and let $\wTheta^u_Y\in \cPC^{\widehat{F}_Y}_G(R^{\ps}_{\Theta_Y})$ be
the universal deformation of $\Theta_Y$.

\begin{lem}\label{thru} The 
map $\Fhat_Y \rightarrow \Gamma/\Ker(\wTheta^u)$ factors as
$\Fhat_Y/\Ker(\wTheta^u_Y)\rightarrow \Gamma /\Ker(\wTheta^u)$. 
\end{lem}
\begin{proof} The same proof as \Cref{factor_thru}.
\end{proof}

\begin{prop}\label{main_kappa} The group $K/\Ker(\wTheta^u)$ is pro-$p$ and 
$R^{\ps, \Gamma}_{\Theta}=R^{\ps, Q}_{\Theta}$.
\end{prop}
\begin{proof} Since $K_Y/\Ker(\wTheta^u_Y)$ is pro-$p$ by 
\Cref{pro-p-again},
the assertion follows from \Cref{rewrite} applied 
with $N=\Ker(\wTheta^u)$. \Cref{thru} implies that the 
assumptions in \Cref{rewrite} are satisfied.
\end{proof}

\begin{cor} If $\Gamma$ satisfies Mazur's 
condition $\Phi_p$ then $R^{\ps, \Gamma}_{\Theta}$ 
is naturally isomorphic to the $\qq$-adic completion
of $\Lambda\otimes_{\OO'} R^{\ps, \Gamma}_{\Thetabar}$. In particular, $R^{\ps, \Gamma}_{\Theta}$ is noetherian.
\end{cor}
\begin{proof} We have $R^{\ps, \Gamma}_{\Thetabar}= R^{\ps, Q}_{\Thetabar}$
by \Cref{same_ring} and $R^{\ps, \Gamma}_{\Theta}=R^{\ps, Q}_{\Theta}$
by \Cref{main_kappa}. Since $\Gamma$ satisfies $\Phi_p$, $Q$ is finitely generated 
and the assertion follows from \Cref{fg_kappa}.
\end{proof}

\section{The rigid analytic space of $G$-pseudocharacters}
Let $\Rig_L$ be the category of rigid analytic spaces over $L$. 
If $Y\in \Rig_L$ then its ring of global sections $\OO(Y)$ 
carries a natural Hausdorff topology. If $Y=\Sp(A)$ is affinoid then the topology on $A$ is the Banach space topology. 
 Let $\tilde{X}_G: \Rig_L^{\op} \rightarrow \Set$ be the functor that associates to every rigid analytic space $Y\in \Rig_L$ the 
set of continuous $G$-pseudocharacters $\cPC^{\Gamma}_G(\OO(Y ))$. We show that
if $\Gamma$ satisfies Mazur's condition $\Phi_p$, then $\tilde{X}_G$ is represented
by a rigid analytic space. 

\begin{prop}\label{sweet} Let $\Theta\in \cPC^{\Gamma}_G(A)$, where $A$ is an affinoid $L$-algebra.
Then $\Gamma/\Ker(\Theta)$ is an extension of a finite group by a pro-$p$ group.
\end{prop}
\begin{proof} Let $\pp_1, \ldots, \pp_r$ be the associated primes of $A$. For 
each $i$ 
we choose a maximal ideal $\mm_i$ containing $\pp_i$. The natural map 
$A\rightarrow \prod_{i=1}^r A_{\mm_i}$ is injective since the map 
$A\rightarrow \prod_{i=1}^r A_{\pp_i}$ is injective. Thus 
the map $A\rightarrow \prod_{i=1}^r \hat{A}_{\mm_i}$ is injective.

Let $\kappa_i$ be the residue field of $\mm_i$. Since $A$ is an affinoid algebra, 
$\kappa_i$ is a finite extension of $\Qp$. Let $\Theta_i=\Theta\otimes_A \kappa_i$ 
then $\Theta_i\in \cPC^{\Gamma}_G(\kappa_i)$ and $\wTheta_i=\Theta\otimes_A \hat{A}_{\mm_i}$
is a deformation of $\Theta_i$ to $\hat{A}_{\mm_i}$. \Cref{main_kappa}
implies that $\Gamma/\Ker(\wTheta_i)$ is an extension of  a finite group
by a pro-$p$ group. Hence the same applies to $\Gamma/\bigcap_{i=1}^r \Ker(\wTheta_i)$.
Since $A$ injects into $\prod_{i=1}^r \hat{A}_{\mm_i}$ we conclude
that $\Gamma/\Ker(\Theta)$ is a quotient of $\Gamma/\bigcap_{i=1}^r \Ker(\wTheta_i)$
and the assertion follows.
\end{proof}

\begin{cor} Let $\Theta\in \cPC^{\Gamma}_G(A)$, where $A$ is an affinoid $L$-algebra.
If $\Gamma$ satisfies Mazur's condition $\Phi_p$, then 
there exists an  $\OO$-subalgebra $A_0 \subseteq A$ which is a quotient of $\OO \langle x_1, \dots, x_t \rangle$ and $A_0[1/p] = A$, such that $\Theta$ takes values in $A_0$.
\end{cor}
\begin{proof} Our assumption on $\Gamma$ and \Cref{sweet} imply that 
$\Gamma/\Ker(\Theta)$ is finitely generated. The assertion follows from 
\cite[Lemma 6.16 (2)]{quast}.
\end{proof}

Let $|\PC^{\Gamma}_G|\subset \PC^{\Gamma}_G$ be the subset of closed points $z$ with finite residue field $k_z$, such that the canonical $G$-pseudocharacter 
$\Theta_z\in \PC^{\Gamma}_G(k_z)$ attached to $z$ is continuous for the discrete topology on $k_z$. Let $\OO_z= \OO\otimes_{W(k)} W(k_z)$ and 
let $R^{\ps}_{\Theta_z}$ be the deformation ring representing 
$\Def_{\Theta_z}: \Aa_{\OO_z}\rightarrow \Set$ and let $\Theta^u_z\in \cPC^{\Gamma}_G(R^{\ps}_{\Theta_z})$ be the universal deformation of $\Theta_z$.
If $\Gamma$ satisfies Mazur's condition $\Phi_p$ then 
$R^{\ps}_{\Theta_z}$ is a complete local noetherian $\OO_z$-algebra 
with residue field $k_z$ by \Cref{Phi_p_noeth}, and we let $(\Spf R^{\ps}_{\Theta_z})^{\rig}$
be the Berthelot rigid generic fibre of the formal scheme $\Spf R^{\ps}_{\Theta_z}$, 
\cite[Section 7]{deJong}.

\begin{thm}\label{main_rigid} If $\Gamma$ satisfies Mazur's  condition $\Phi_p$, then 
the functor $\tilde{X}_G$ is representable by a quasi-Stein space
$ X_G:=\coprod_{z\in |\PC^{\Gamma}_G|} (\Spf R^{\ps}_{\Theta_z})^{\rig}$.
\end{thm} 
\begin{proof} Let $(R^{\ps}_{\Theta_z})^{\rig}$ be the ring of global sections
of $(\Spf R^{\ps}_{\Theta_z})^{\rig}$. Then $\OO(X_G)=\prod_z (R^{\ps}_{\Theta_z})^{\rig}$ and $\Theta^{\univ}:= \prod_z \Theta^u_z\otimes_{R^{\ps}_{\Theta_z}} (R^{\ps}_{\Theta_z})^{\rig} \in 
\tilde{X}_G(X_G)$. A morphism of rigid analytic spaces $Y\rightarrow X_G$ induces 
a map of topological $L$-algebras $\OO(X_G) \rightarrow \OO(Y)$ to which 
we may associate $\Theta^{\univ}\otimes_{\OO(X_G)} \OO(Y)\in \tilde{X}_G(Y)$.
We claim that the map $\Mor_{\Rig_L}(Y, X_G)\rightarrow \tilde{X}_G(Y)$ defined above is 
bijective. It is enough to prove the claim for $Y$ affinoid, since the 
general case follows by considering admissible coverings. 

Let $A$ be an $L$-affinoid algebra, 
and let $\Theta\in \cPC^{\Gamma}_G(A)$. It follows from 
\Cref{sweet} that $\Theta\in \cPC^{Q}_G(A)$ for a finitely generated 
quotient $Q$ of $\Gamma$. Since the assertion of the theorem for finitely 
generated groups $\Gamma$ is proved in \cite[Theorem 6.1]{quast}, $\Theta$ 
induces a morphism of rigid analytic spaces:
\begin{equation} 
\Sp(A) \rightarrow \coprod_{z\in |\PC^{Q}_G|} (\Spf R^{\ps, Q}_{\Theta_z})^{\rig}
\subseteq \coprod_{z\in |\PC^{\Gamma}_G|} (\Spf R^{\ps, \Gamma}_{\Theta_z})^{\rig}
\end{equation}
such that $\Theta= \Theta^{\univ}\otimes_{\OO(X_G)} A$. Thus we 
obtain a canonical bijection between $\tilde{X}_G(\Sp(A))$ and 
$\Mor_{\Rig_L}(\Sp(A), X_G)$.  
\end{proof} 

\begin{cor}\label{points_rigid} If $\Gamma$ satisfies Mazur's condition $\Phi_p$ then the $\Lbar$-points
of $X_G$ are in canonical bijection with $G^0(\Lbar)$-conjugacy classes of continuous $G$-semisimple representations $\Gamma\rightarrow G(\Lbar)$.
\end{cor}
\begin{proof} If $\Gamma$ is finitely generated then the assertion 
is proved in \cite[Corollary 6.23]{quast} and the same proof carries over. 
By the definition of $X_G$, we have $X_G(\Lbar) = \cPC^{\Gamma}_G(\Lbar)$, where
$\Lbar$ carries the direct limit topology of its finite subextensions of $L$. The claim now follows from the reconstruction theorem \cite[Theorem 3.8]{quast}.
\end{proof} 

\begin{remar} A priori it is not at all clear that the set of $G^0(\Lbar)$-conjugacy classes of continuous $G$-semisimple representations $\rho:\Gamma\rightarrow G(\Lbar)$
can be parameterised by $\Lbar$-points of a rigid analytic space. The finiteness 
statements proved in \Cref{Phi_p_noeth} and \Cref{main_rigid} are key inputs 
in the proof of this statement. 
\end{remar}

\bibliographystyle{plain}
\bibliography{Ref}

@incollection {Gouvea,
    AUTHOR = {Gouv\^{e}a, Fernando Q.},
     TITLE = {Deformations of {G}alois representations},
 BOOKTITLE = {Arithmetic algebraic geometry ({P}ark {C}ity, {UT}, 1999)},
    SERIES = {IAS/Park City Math. Ser.},
    VOLUME = {9},
     PAGES = {233--406},
      NOTE = {Appendix 1 by Mark Dickinson, Appendix 2 by Tom Weston and
              Appendix 3 by Matthew Emerton},
 PUBLISHER = {Amer. Math. Soc., Providence, RI},
      YEAR = {2001},
   MRCLASS = {11F80 (11F85 11R39 14D15)},
  MRNUMBER = {1860043},
MRREVIEWER = {Gebhard B\"{o}ckle},
       DOI = {10.1090/pcms/009/05},
       URL = {https://doi.org/10.1090/pcms/009/05},
}

@incollection {che_durham,
    AUTHOR = {Chenevier, Ga\"{e}tan},
     TITLE = {The {$p$}-adic analytic space of pseudocharacters of a
              profinite group and pseudorepresentations over arbitrary
              rings},
 BOOKTITLE = {Automorphic forms and {G}alois representations. {V}ol. 1},
    SERIES = {London Math. Soc. Lecture Note Ser.},
    VOLUME = {414},
     PAGES = {221--285},
 PUBLISHER = {Cambridge Univ. Press, Cambridge},
      YEAR = {2014},
   MRCLASS = {11F70 (11F80 14G22 20E18)},
  MRNUMBER = {3444227},
MRREVIEWER = {Cameron Franc},
}

@article {WE_alg,
    AUTHOR = {Wang-Erickson, Carl},
     TITLE = {Algebraic families of {G}alois representations and potentially
              semi-stable pseudodeformation rings},
   JOURNAL = {Math. Ann.},
  FJOURNAL = {Mathematische Annalen},
    VOLUME = {371},
      YEAR = {2018},
    NUMBER = {3-4},
     PAGES = {1615--1681},
      ISSN = {0025-5831},
   MRCLASS = {11F80 (11S20 14D15 14L24)},
  MRNUMBER = {3831282},
MRREVIEWER = {Meng Fai Lim},
       DOI = {10.1007/s00208-017-1557-8},
       URL = {https://doi-org.ezproxy.library.ubc.ca/10.1007/s00208-017-1557-8},
}

@misc{stacks-project,
    shorthand    = {Stacks},
    author       = {The {Stacks Project Authors}},
    title        = {\textit{Stacks Project}},
    howpublished = {\url{https://stacks.math.columbia.edu}},
    year         = {2022},
}

@Article{deJong,
 Author = {de Jong, A. J.},
 Title = {Crystalline {Dieudonn{\'e}} module theory via formal and rigid geometry},
 FJournal = {Publications Math{\'e}matiques},
 Journal = {Publ. Math., Inst. Hautes {\'E}tud. Sci.},
 ISSN = {0073-8301},
 Volume = {82},
 Pages = {5--96},
 Year = {1995},
 Language = {English},
 DOI = {10.1007/BF02698637},
 URL = {https://eudml.org/doc/104109},
 zbMATH = {910524},
 Zbl = {0864.14009}
}

@article{finite, 
author={Pa\v{s}k{\=u}nas, Vytautas and Tung, Shen-Ning},
title={Finiteness properties of the category of mod $p$ representations of $\mathrm{GL}_2(\mathbb{Q}_p)$},
year={2021},
howpublished ={\url{https://arxiv.org/pdf/2104.08948.pdf}}
}

@article {seshadri,
    AUTHOR = {Seshadri, C. S.},
     TITLE = {Geometric reductivity over arbitrary base},
   JOURNAL = {Adv.\,Math.},
  FJOURNAL = {Advances in Mathematics},
    VOLUME = {26},
      YEAR = {1977},
    NUMBER = {3},
     PAGES = {225--274},
      ISSN = {0001-8708},
   MRCLASS = {14L99 (14D20 15A72)},
  MRNUMBER = {466154},
MRREVIEWER = {Vladimir L. Popov},
       DOI = {10.1016/0001-8708(77)90041-X},
       URL = {https://doi.org/10.1016/0001-8708(77)90041-X},
}

@article {alper,
    AUTHOR = {Alper, Jarod},
     TITLE = {Adequate moduli spaces and geometrically reductive group
              schemes},
   JOURNAL = {Algebr. Geom.},
  FJOURNAL = {Algebraic Geometry},
    VOLUME = {1},
      YEAR = {2014},
    NUMBER = {4},
     PAGES = {489--531},
      ISSN = {2313-1691},
   MRCLASS = {14D20 (14A20 14L15 14L24 14L30)},
  MRNUMBER = {3272912},
MRREVIEWER = {Fabio Tonini},
       DOI = {10.14231/AG-2014-022},
       URL = {https://doi.org/10.14231/AG-2014-022},
}

@article {BIP_new,
    AUTHOR = {B\"{o}ckle, Gebhard and Iyengar, Ashwin and Pa\v{s}k\={u}nas,
              Vytautas},
     TITLE = {On local {G}alois deformation rings},
   JOURNAL = {Forum Math. Pi},
  FJOURNAL = {Forum of Mathematics. Pi},
    VOLUME = {11},
      YEAR = {2023},
     PAGES = {Paper No. e30, 54},
      ISSN = {2050-5086},
   MRCLASS = {11F80 (11F85)},
  MRNUMBER = {4660960},
       DOI = {10.1017/fmp.2023.25},
       URL = {https://doi.org/10.1017/fmp.2023.25},
}

@article {BMR,
    AUTHOR = {Bate, Michael and Martin, Benjamin and R\"{o}hrle, Gerhard},
     TITLE = {A geometric approach to complete reducibility},
   JOURNAL = {Invent. Math.},
  FJOURNAL = {Inventiones Mathematicae},
    VOLUME = {161},
      YEAR = {2005},
    NUMBER = {1},
     PAGES = {177--218},
      ISSN = {0020-9910,1432-1297},
   MRCLASS = {20G15 (20E42)},
  MRNUMBER = {2178661},
MRREVIEWER = {Andy\ R.\ Magid},
       DOI = {10.1007/s00222-004-0425-9},
       URL = {https://doi.org/10.1007/s00222-004-0425-9},
}

@article {BHKT,
    AUTHOR = {B\"{o}ckle, Gebhard and Harris, Michael and Khare,
              Chandrashekhar and Thorne, Jack A.},
     TITLE = {{$\hat G$}-local systems on smooth projective curves are
              potentially automorphic},
   JOURNAL = {Acta Math.},
  FJOURNAL = {Acta Mathematica},
    VOLUME = {223},
      YEAR = {2019},
    NUMBER = {1},
     PAGES = {1--111},
      ISSN = {0001-5962,1871-2509},
   MRCLASS = {11F80 (11F70 14D24)},
  MRNUMBER = {4018263},
MRREVIEWER = {Gabor\ Wiese},
       DOI = {10.4310/ACTA.2019.v223.n1.a1},
       URL = {https://doi.org/10.4310/ACTA.2019.v223.n1.a1},
}

@article{DHKM,
 author = {Dat, Jean-Fran{\c{c}}ois and Helm, David and Kurinczuk, Robert and Moss, Gilbert},
 title = {Moduli of {Langlands} parameters},
 fjournal = {Journal of the European Mathematical Society (JEMS)},
 journal = {J. Eur. Math. Soc. (JEMS)},
 issn = {1435-9855},
 volume = {27},
 number = {5},
 pages = {1827--1927},
 year = {2025},
 language = {English},
 doi = {10.4171/JEMS/1599},
 zbMATH = {8029940}
}

@misc{quast,
      title={Deformations of {$G$}-valued pseudocharacters}, 
      author={Quast, Julian},
      year={2023},
      eprint={2310.14886},
      archivePrefix={arXiv},
      primaryClass={math.NT},
      howpublished ={\url{https://arxiv.org/abs/2310.14886}}
}

@article{Laf,
  author =	 {Lafforgue, Vincent},
  title =	 {Chtoucas pour les groupes r\'eductifs et param\'etrisation de {L}anglands globale},
  journal =	 {J. Amer. Math. Soc.},
  year =	 2018,
  volume =	 31,
  number =	 3,
  pages =	 {719-891}
}

@article {cotner,
    AUTHOR = {Cotner, Sean},
     TITLE = {Morphisms of character varieties},
   JOURNAL = {Int. Math. Res. Not. IMRN},
  FJOURNAL = {International Mathematics Research Notices. IMRN},
      YEAR = {2024},
    NUMBER = {16},
     PAGES = {11540--11548},
      ISSN = {1073-7928,1687-0247},
   MRCLASS = {14M35 (14L15 20G35)},
  MRNUMBER = {4789091},
       DOI = {10.1093/imrn/rnae124},
       URL = {https://doi.org/10.1093/imrn/rnae124},
}

@article {BJ_new,
    AUTHOR = {B\"{o}ckle, Gebhard and Juschka, Ann-Kristin},
     TITLE = {Equidimensionality of universal pseudodeformation rings in
              characteristic {$p$} for absolute {G}alois groups of
              {$p$}-adic fields},
   JOURNAL = {Forum Math. Sigma},
  FJOURNAL = {Forum of Mathematics. Sigma},
    VOLUME = {11},
      YEAR = {2023},
     PAGES = {Paper No. e102},
      ISSN = {2050-5094},
   MRCLASS = {11F80 (11F70 11F85)},
  MRNUMBER = {4668540},
       DOI = {10.1017/fms.2023.82},
       URL = {https://doi.org/10.1017/fms.2023.82},
}

@misc{ScholzeCond,
    title={Lectures on condensed mathematics},
    author={Peter Scholze},
    year={2019},
howpublished ={\url{https://www.math.uni-bonn.de/people/scholze/Condensed.pdf}}

     }

@misc{emerson2023comparison,
      title={Comparison of different definitions of pseudocharacters}, 
      author={Kathleen Emerson and Sophie Morel},
      year={2023},
      eprint={2310.03869},
      archivePrefix={arXiv},
      primaryClass={math.AG},
howpublished ={\url{https://arxiv.org/abs/2310.03869}}
}

@article{defG,
 Author = {Pa{\v{s}}k{\=u}nas, Vytautas and Quast, Julian},
 Title = {On local {Galois} deformation rings: generalised reductive groups},
 Year = {2024},
 
 note = {\url{https://arxiv.org/abs/2404.14622}},
 arXiv = {arXiv:2404.14622}
}

@article{inf_laf,
Author = {Pa{\v{s}}k{\=u}nas, Vytautas and Quast, Julian},
 Title = {Infinitesimal characters and {L}afforgue's pseudocharacters},
 Year = {2025},
 
note = {arXiv preprint}
}

@article {weidner,
    AUTHOR = {Weidner, M.},
     TITLE = {Pseudocharacters of homomorphisms into classical groups},
   JOURNAL = {Transform. Groups},
  FJOURNAL = {Transformation Groups},
    VOLUME = {25},
      YEAR = {2020},
    NUMBER = {4},
     PAGES = {1345--1370},
      ISSN = {1083-4362,1531-586X},
   MRCLASS = {20C99 (20G05)},
  MRNUMBER = {4166690},
MRREVIEWER = {Matthias\ Gr\"uninger},
       DOI = {10.1007/s00031-020-09603-2},
       URL = {https://doi.org/10.1007/s00031-020-09603-2},
}

@incollection {zhu_coherent,
    AUTHOR = {Zhu, Xinwen},
     TITLE = {Coherent sheaves on the stack of {L}anglands parameters},
 BOOKTITLE = {The {L}anglands {P}rogram},
    SERIES = {Proc. Sympos. Pure Math.},
    VOLUME = {112.2},
     PAGES = {39--123},
 PUBLISHER = {Amer. Math. Soc., Providence, RI},
      YEAR = {2025},
      ISBN = {9781470484644; 9781470474386},
   MRCLASS = {99-06},
  MRNUMBER = {5007750},
       DOI = {10.1090/pspum/112.2/02064},
       URL = {https://doi.org/10.1090/pspum/112.2/02064},
}

@misc{jinyue, 
author={Luo, Jinyue},
title={Pseudorepresentations not arising from genuine representations},
note={\url{https://arxiv.org/abs/2310.16953}},
year={2023} 
}
\end{document}